\documentclass[12pt]{amsart}
\usepackage{amsmath,amsthm,latexsym,amscd,amsbsy,amssymb,url}
\usepackage[all]{xypic}
 \relax

\chardef\bslash=`\\ 

\makeatletter
\def\verbatim{\interlinepenalty\@M \@verbatim
  \leftskip\@totalleftmargin\advance\leftskip2pc
  \frenchspacing\@vobeyspaces \@xverbatim}
\makeatother
\newtheorem{thm}{Theorem}[section]
\newtheorem{cor}[thm]{Corollary}
\newtheorem{lem}[thm]{Lemma}

\newtheorem{ex}[thm]{Example}
\newtheorem{que}[thm]{Question}

\begin{document}


\title
{Factorization of Bijections onto Ordered Spaces}
\author{Raushan Buzyakova}
\email{Raushan\_Buzyakova@yahoo.com}
\author{Alex Chigogidze}


\begin{abstract}{
We identify a class of subspaces of  ordered spaces $\mathcal L$ for which the following statement holds:
If $f:X\to L\in \mathcal L$ is a continuous bijections of a zero-dimensional space $X$, then $f$ can be re-routed via a zero-dimensional subspace of an ordered space that has weight not exceeding that of $L$.

}
\end{abstract}

\maketitle
\markboth{R. Buzyakova and A. Chigogidze}{Factorization of Bijections onto Ordered Spaces}
{ }

\section{Introduction}\label{S:intro}

\par\bigskip
The paper studies  factorization of continuous bijections onto subspaces of ordered spaces. In \cite{Mar}, Mardesic proved, in particular, that any continuous map of an $n$-dimensional compactum into a metric compact space can be re-routed via a metric compactum of dimension at most $n$. Deep research has been inspired by this result. We would like to mention a theorem of Pasynkov \cite{Pas} that any continuous map of  a Tychonoff Ind-$n$ space into a metric space $M$ admits a factorization with the middle space  of Ind-dimension at most $n$ and weight at most $w(M)$. This result implies, in particular, that if a space of Ind-dimension $n$  admits a continuous injection into a metric space, then it admits a continuous injection into a metric space of Ind-dimension at most $n$.  The paper is devoted to the following general problem: 
\par\medskip\noindent
{\it Problem. Let a space $X$ of dimension $n$ admit a continuous bijection onto a space with property $P$. Does $X$ admit a continuous bijection onto a space of dimension at most $n$ and  with property $P$?}
\par\medskip\noindent
In this paper we are interested in inductive "ind"  dimension.
We first observe that for any $n>2$ there exists a zero-dimensional space $X$ that admits a continuous injection into $\mathbb R^{n}$ but not into $\mathbb R^{n-1}$. We then observe that, if a zero-dimensional space continuously injects into  $\mathbb R$, then it admits a continuous injections into the Cantor Set. The latter observation led the authors to the main result of this paper. To state the result, let $\mathcal L$ be the class of all subspaces of ordered spaces that have a $\sigma$-disjoint $\pi$-base. Note that a subspace  $L$ of an ordered space is in $\mathcal L$ if and only if $L$ has a dense subset $\cup_n X_n$, where each $X_n$ is a discrete (in itself) subspace. In our main result, Theorem \ref{thm:zerodim}, we prove that {\it if  $f:X\to L\in \mathcal L$ is a continuous bijection of a zero-dimensional space $X$, then $f$ can be re-routed via a zero-dimensional subspace $M$ of an ordered space of weight at most that of $L$.}  One may wonder if the class $\mathcal L$ can be replaced by the class of all ordered spaces and their subspaces. 
The authors do not have an example to justify the restriction on the class $\mathcal L$ but believe that there is a good chance for an example.

In notation and terminology, we will follow \cite{Eng}.  Spaces of ind-dimension $0$ are called {\it zero-dimensional}. A linearly ordered space (abbreviated as LOTS)  is one with the topology induced by some order. A subspace of a LOTS is a generalized ordered space (abbreviated as GO).  It is a known and very useful fact that 
 a Hausdorff space $L$ is a GO-space, if  its topology can be generated by a collection of convex sets with respect to some order on $L$ (see, for example, \cite{BL}). A subset $A$ of an ordered set $\langle S,\prec \rangle $ is {\it $\prec$-convex} if it is convex with respect to $\prec$. A collection $\mathcal U$ of non-empty open sets of a space $X$ is {\it a $\pi$-base} of $X$ if for any non-empty open set $O$ in $X$ there exists $U\in \mathcal U$ such that $U\subset O$. All spaces are assumed Tychonoff. For the purpose of readability we will occasionally resort to an informal argument.

\par\bigskip
\section{Motivation}\label{S:factorization}

\par\bigskip
Ample research has been done to describe spaces that admit continuous injections into metric spaces. 
Assume that a space $X$ has a certain dimension and admits a continuous injection into  a metric space. It is natural to wonder if $X$ admits a continuous injection into a metric space of the same dimension. As mentioned in the introduction, if dimension is in the sense of Ind or dim, then the affirmative answer is a simple corollary of the mentioned Pasynkov factorization theorem. In case of  ind-dimension,  additional analysis may be required. We start with examples. 

\par\bigskip\noindent
\begin{ex}\label{ex:example1} 
For each $N>1$
there exists a separable space $X$ of $ind (X)=0$ that admits a continuous injection into  $\mathbb R^N$ but not into $\mathbb R^{N-1}$.
\end{ex}
\par\smallskip\noindent
{\it Construction.} 
Put $D = \{\langle q_1,...,q_N\rangle: q_i\in \mathbb Q\}$ and $P=\mathbb R^N\setminus D$.
The underlying set for our space is $\mathbb R^N=P\cup D$. We will define a new topology $\mathcal T$ on $\mathbb R^N$ so that $X=\langle\mathbb R^N, \mathcal T\rangle$ has desired properties.  When $\mathbb R^N$ is used with the Euclidean topology, we will refer to it by $\mathbb R^N$. In our new topology, points of $D$ are declared isolated.
Base neighborhoods at points of $P$ will be defined in two stages. 

\par\bigskip\noindent
\underline {\it Stage 1.}
For each map $f:D\to \mathbb R^{N-1}$, put 
$$P_f = \{x\in P: f\ cannot\  be\ continuously\ extended\ to \ D\cup\{x\}\},$$
where continuity at $x\in P$ is considered with respect to the Euclidean topology and $D$ is regarded as discrete.  Next let $F$ be the set of  all functions $f$ from $D$ to $\mathbb R^{N-1}$ such that $|P_f| = 2^\omega$.
For each $f\in F$, fix $x_f\in P_f$ so that $x_f\not = x_g$ for distinct $f,g\in F$. This can be done since $|F|=|P_f|=2^\omega$ for each $f\in F$. Let us define a base neighborhood at $x_f$, for a fixed $f\in F$.
Recall that $f$  cannot be continuously extended to $x_f$ with respect to the Euclidean topology at $x_f$. Therefore, there exists a sequence 
$\langle x_{f,n}\rangle_n$ of elements of $D$ that converges to $x_f$ in $\mathbb R^N$ such that $\langle f(x_{f,n})\rangle_n$ is not converging in $\mathbb R^{N-1}$. Sets in form $\{x_f\}\cup \{x_{f,i}: i > n\}$ will be base neighborhoods at $x_f$ in our new topology. This completes Stage 1.

\par\bigskip\noindent
\underline{\it Stage 2.}
At the first stage we may not have defined base neighborhoods at all points of $P$. If $x\in P$ is such a point, fix an arbitrary sequence $\langle x_{n}\rangle_n$ of elements of $D$ that converges to $x$ in $\mathbb R^N$. A base neighborhood at $x$ is in form $\{x\}\cup \{x_{i}: i > n\}$. This completes Stage 2 and definition of base neighborhoods at all points of $X$.

\par\bigskip
The topology $\mathcal T$ is generated  by the identified base neighborhoods. The space $X=\langle \mathbb R^N,\mathcal T \rangle$  has ind-dimension $0$, which follows from the definition of base neighborhoods. The Euclidean topology of $\mathbb R^N$ is a subtopology of $X$. Hence, $X$ admits a continuous bijection onto $\mathbb R^N$. Let us show that $X$ does not admit a continuous injection into $\mathbb R^{N-1}$. Fix any continuous function $h: X\to \mathbb R^{N-1}$. Assume that $f=h|_D$ is one-to-one. If $P_f$  had the cardinality of continuum, then 
$f$ would  have been discontinuous at $x_f$ due to Stage 1. Therefore, $P_f$ has a smaller cardinality. Then there exists  an irrational number $r^*$ such that $L=\{\langle r^*,r_2,...,r_N\rangle :0\leq r_i\leq 1\}\subset P\setminus (P_f\cup D)$. By the definition of $P_f$, for every 
$x\in P\setminus P_f$,  the function $h|_{D\cup \{x\}}$ is continuous at $x$ with respect to the Euclidean topology. This fact and density of $D$ in both $X$ and $\mathbb R^N$  imply the following claim. 
\par\medskip\noindent
{\it Claim. $h|_L$ is continuous on $L$ with respect to the Euclidean topology of $L$.}
\par\medskip\noindent
The set $L$ with the Euclidean topology is homeomorphic to a closed disc of $\mathbb R^{N-1}$.
Since  $h|_L$ is one-to-one on $L$ and $L$ is compact with respect to the Euclidean topology,  by Claim,  $h(L)$ has non-empty interior in $\mathbb R^{N-1}$.
Since $D$ is dense in $X$, we conclude that $h(d) \in h(L)$ for some $d\in D$.  Since $d\not\in L$, we conclude that  $h$ is not one-to-one on $X$.
\begin{flushright}$\square$\end{flushright}

\par\bigskip\noindent
Example \ref{ex:example1} may lead to an impression that with additional efforts one may construct a zero-dimensional space
which is continuously injectable into $\mathbb R$ but not into the Cantor Set.
However, it is not the case as observed next.  
\par\bigskip\noindent
\par\bigskip\noindent
\begin{thm}\label{thm:toR}
Let $f$ be a continuous injection of a zero-dimensional space $X$ into $\mathbb R$. Then $X$ can be continuously injected into the Cantor Set.
\end{thm}
\begin{proof}
Fix a countable dense subset $A$ of $f(X)$. We may assume that $a,b\in A$ whenever $a$ and $b$ are immediate neighbors in $f(X)$.
For each $a\in A$ and a positive integer $n$, fix a clopen neighborhood $U_{a,n}$ of $f^{-1}(a)$ such that
$f(U_{a,n})\subset (a-1/n, a+1/n)$. This a possible since  $X$ is zero-dimensional.

Define $\mathcal V$  as follows. A set $V$ is in $\mathcal V$ if and only if it falls into one of the following three categories:
\begin{enumerate}
	\item $V=U_{a,n}$ for some $a\in A$ and a positive integer $n$.
	\item $V=f^{-1}((-\infty, a])\setminus U_{a,n}$ for some $a\in A$ and a positive integer $n$.
	\item $V=f^{-1}([a,\infty ))\setminus U_{a,n}$ for some $a\in A$ and a positive integer $n$.
\end{enumerate}
The conclusion of the theorem follows from the following two claims.
\par\medskip\noindent
{\it Claim 1. Each element of $\mathcal V$ is a clopen subset of $X$.}
\par\smallskip\noindent
{\it Proof of Claim 1.} To prove Claim 1, it is enough to show that sets in category 2 are clopen. Each such set is closed as a difference of a closed set and an open set. Let us show open-ness.  Since $f^{-1}(a)$ is  a subset of
$f^{-1}((-\infty, a])\cap U_{a,n}$, we conclude that
$f^{-1}((-\infty, a])\setminus U_{a,n} =f^{-1}((-\infty, a))\setminus U_{a,n} $. The right hand side of the equality is open as the difference of open and closed sets. The claim is proved.
\par\medskip\noindent
{\it Claim 2. $\mathcal V$ is a countable point-separating family.}
\par\smallskip\noindent
{\it Proof of Claim 2.} Countability follows from countability of $A$ and the set of integers.  Let $x$ and $y$ be two distinct points of $X$.  Since $f$ is one-to-one, we may assume that $f(x)<f(y)$.
 Recall that $A$ is dense in $f(X)$ and contains any two points that are  immediate neighbors in $f(X)$. Therefore, there exists $a\in A$ such that $f(x)\leq a< f(y)$ or $f(x)<a\leq f(y)$.

Assume that $f(x)<a<f(y)$. Let $n$ be a positive integer such that $f(x)<a-1/n$.
Then $x$ is an element of $f^{-1}((-\infty, a])\setminus U_{a,n}$ and $y$ is not.

Assume that $f(x)=a<f(y)$.  Let $n$ be a positive integer such that $a+1/n<f(y)$. Then $x$ is an element of $U_{a,n}$ and $y$ is not. The case $f(x)<a=f(y)$ is treated similarly.
\end{proof}
\par\bigskip\noindent
Note that the zero-dimensional  topology constructed in the proof of the above theorem need not be comparable with the topology on $X$ determined by the given map into $\mathbb R$.  Let $\mathcal T$ denote the topology generated by the union of the these two subtopologies. Clearly, $Y=\langle X, \mathcal T\rangle$ is a zero-dimensional separable metrizable space. Thus, the following factorization version of Theorem \ref{thm:toR} holds.
\par\bigskip\noindent
\begin{cor}\label{cor:toR}
Let $X$ be zero-dimensional and let $f:X\to \mathbb R$ be a continuous injection. Then there exist a subspace $Y$ of the Cantor Set and continuous maps $g:X\to Y$ and $h:Y\to \mathbb R$ such that $f=h\circ g$.
\end{cor}  
\par\bigskip
It is natural to wonder if the target space $\mathbb R$, in statements  \ref{thm:toR} and \ref{cor:toR},  can be replaced by any (generalized) ordered space. This leads to our main result presented in the next section.
\par\bigskip

\par\bigskip
\section{Factorization Theorem}\label{S:factorization}

\par\bigskip
In this section, we generalize statements \ref{thm:toR} and \ref{cor:toR} for ranges from a sufficiently wide class. As demonstrated in the previous section, factorization version (Corollary \ref{cor:toR} )  is an effortless corollary of  injectability into the Cantor Set (Theorem \ref{thm:toR}) if the target space is the reals . In general, a topology generated by the union of two GO-topologies need not be  a GO-topology (see \cite{Reed} for examples). We, therefore, prove factorization version directly. We will generalize our statements to GO-spaces from the following class.

\par\bigskip\noindent
{\bf Definition of class $\mathcal L$.} {\it A generalized ordered space $L$ is in $\mathcal L$ if and only if $L$ has a dense subset which is a countable union of discrete (in themselves) subspaces.}

\par\bigskip\noindent
Another way to define this class is as follows.

\par\bigskip\noindent
{\bf Another Definition of class $\mathcal L$.} {\it
A generalized ordered space $L$ is in $\mathcal L$ if and only if $L$ has a $\sigma$-disjoint $\pi$-base.
}

\par\bigskip\noindent
One can easily show that all separable GO-spaces as well as all subspaces of any ordinal are in the class. The long line as well as the lexicographical product of any ordinal and $[0,1)$ are in the class too. Note that the class need not contain subspaces of its members. In our proofs, we will be using the second definition of the class $\mathcal L$. For completeness, we next prove that the definitions are indeed equivalent.

\par\bigskip\noindent
\begin{lem}\label{lem:equivalentdef}
A GO-space $L$ has a $\sigma$-disjoint $\pi$-base if and only if $L$ has a dense subspace that is a countable union of discrete subspaces.
\end{lem}
\begin{proof}
($\Rightarrow$). Fix a $\sigma$-disjoint $\pi$-base and then select a point from each element of the base. The set of selected points is as desired.
\par\medskip\noindent
($\Leftarrow$) Let $A_0$ be the set of all isolated points of $L$ and $M=L\setminus \overline A_0$.
If $M$ is empty, then $\{\{a\}:a\in A_0\}$ is a desired $\pi$-base. We now assume that $M$ is not empty.

 The set $M$ has no isolated points. By Lemma's hypothesis, for any $n\geq 1$, we can find $A_n$, a discrete (in itself) subspace of $M$, so that $\cup_n A_n$ is dense in $M$.

 Put
$\mathcal I_0 = \{\{a\}:a\in A_0\}$.
Next, for each $i,j>0$, define $\mathcal I_{ij}$ as follows: $I\in \mathcal I_{ij}$ if and only if $I$ is a maximal non-empty convex subset of $M\setminus (A_i\cup A_j)$ that is open in $L$. By maximality, $\mathcal I_{ij}$ is a disjoint family of non-empty open convex subsets of $L$.

Let us show that $\mathcal I_0 \cup [\bigcup_{i,j}\mathcal I_{ij}]$ is a $\pi$-base for $L$. Fix any non-empty open subset $U$ of $L$. We may assume that $U$ is convex. If $U$ has an isolated point $a$, then  $a\in A_0$. We then  have $\{a\}\in \mathcal I_0$ and $\{a\}\subset U$. Assume that $U$ does not have isolated points. Since  $\cup_n A_n$ is dense in $M$, we can find three distinct points $a,b$ and $c$ in $\cup_n A_n$ such that $a,c\in U$ and $a<b<c$. Fix indices $i$ and $j$ such that
$a\in A_i$  and $c\in A_j$.  Since both $A_i$ and $A_j$   are discrete in themselves and $M$ has no isolated points, we conclude that $A_i\cup A_j$ is nowhere dense in $U$. Since $a<b<c$, the interval $(a,c)_L$ is a non-empty open subset of $U$ that has no isolated points. 
Therefore, there exists $I\in \mathcal I_{ij}$ such that $I$ is between $a$ and $c$. Since $U$ is convex, $I$ is a subset of $U$.
\end{proof}

\par\bigskip\noindent
The following folklore-type statement will be used in the main result and is proved  for completeness.
\par\bigskip\noindent
\begin{lem}\label{lem:limitorder}
Let $R_n$ be an order relation on $S$ for each $n$. Suppose that for any distinct $x,y\in S$ there exists $n_{xy}$ such that
either $xR_ny$ for all $n>n_{xy}$ or $yR_nx$ for all $n>n_{xy}$. Then $R$ is a linear order on $X$, where  
$$
R=\{\langle x,y\rangle: xR_ny \ for\  all\ n>n_{xy}\}
$$
\end{lem}
\begin{proof}
Since every element of $R$ is an element of $R_n$ for some $n$, we conclude that $\langle x,x\rangle\not\in R$ for any $x$. Therefore, $R$ is irreflexive. To show comparability, fix any distinct $x,y\in S$. By Lemma's hypothesis, we may assume that $\langle x,y\rangle \in R_n$ for all $n>n_{xy}$. By the definition of $R$, $\langle x,y\rangle \in R$. To show transitivity, fix $\langle x,y\rangle, \langle y,z\rangle\in R$. Let $m=\max\{n_{xy},n_{yz}\}$. Then, $\langle x,y\rangle, \langle y,z\rangle\in R_n$ for all $n>m$. By transitivity of $R_n$, we conclude that $\langle x,z\rangle\in R_n$ for all $n>m$. Therefore, $\langle x,z\rangle\in R$.
\end{proof}

\par\bigskip
We are ready to prove our main result.

\par\bigskip\noindent
\begin{thm}\label{thm:zerodim}
Let $X$ be a zero-dimensional space and $L\in \mathcal L$. If $f:X\to L$ is a continuous bijection, then there exist a        zero-dimensional GO-space $M$ of weight at most that of $L$ and continuous bijections $g:X\to M$ and $h:M\to L$ such that $f=h\circ g$.
\end{thm}
\begin{proof} 
Fix a $\pi$-base  $\bigcup_n \mathcal I_n$  of $L$, where each $\mathcal I_n$ is a disjoint family of convex subsets of $L$. We may assume that $\mathcal I_0=\emptyset$.  We may also assume that the ground set for $X$ and $L$ is the same set $S$. Thus, $X=\langle S, \mathcal T_X\rangle$ and $L=\langle S,<_L,\mathcal T_L\rangle$. For simplicity we may write $<$ instead of $<_L$.

For each $n\in \omega$ we will define families to be later used in the  construction of our middle space $M$.

\par\bigskip\noindent
{\bf Step $0$.} Put 
$\mathcal A_0 =\emptyset$,  $\mathcal B_0=<_L$,  $<_0=\mathcal A_0\cup \mathcal B_0$, and $\mathcal T_0=\mathcal T_L$.

\par\medskip\noindent
{\bf Assumption.} Assume that for each $k=0,...,n-1$ we have defined relations $\mathcal A_k$ and $\mathcal B_k$ on $S$ and a topology $\mathcal T_k$ on $S$ so that the following properties are satisfied:
\begin{description}
	\item[\rm P1] $\mathcal A_m\subset \mathcal A_k$ whenever $m<k$.
	\item[\rm P2] $\mathcal B_m\supset \mathcal B_k$ whenever $m<k$.
	\item[\rm P3] $<_k=\mathcal A_k\cup \mathcal B_k$ is a linear order relation on $S$.
	\item[\rm P4] If $I\in \mathcal I_k$, then $I$ contains a non-empty set which is  clopen with respect to $\mathcal T_k$.	
	\item[\rm P5] If $U$ is $<_m$-convex and  clopen in $\langle S,\mathcal T_m\rangle$, then $U$ is $<_k$-convex and clopen  in $\langle  S, \mathcal T_k\rangle$ whenever $m<k$.
	\item[\rm P6] $\langle S, <_k, \mathcal T_k\rangle$ is a GO-space.
	\item[\rm P7] $\mathcal T_m\subset \mathcal T_k\subset \mathcal T_X$ whenever $m\leq k$.
	\item[\rm P8] The weight of $\langle S,\mathcal T_k\rangle$ is equal to the weight of $L$.
\end{description}
\par\medskip\noindent
Note that the conditions are met for $k=0$.

\par\medskip\noindent
{\bf  Step $n>0$.} Let $\mathcal C_n$ be the collection  of maximal   subsets of $\langle S, <_{n-1}, \mathcal T_{n-1}\rangle$ that are open, connected, and without end-points.
For each $C\in \mathcal C_{n}$ and $I\in \mathcal I_n$, such that $C\cap I\not = \emptyset$, select $y_{CI}\in C\cap I$. By P7, $C\cap I$ is open in $X$. Since $X$ is zero-dimensional, there exists  $V_{CI}\subset C\cap I$ such that $y_{CI}\in V_{CI}$ and $V_{CI}$ is clopen in $X$.

\par\medskip\noindent
Put 
$$
L_{CI}=\{x\in  (C\cap I)\setminus V_{CI}: x<_{n-1}y_{CI}\}\ {\rm and}\  R_{CI}=\{x\in (C\cap  I)\setminus V_{CI}: y_{CI}<_{n-1} x\}.
$$

\par\medskip\noindent
Define $\mathcal A'_n$ as follows.  A pair $\langle x,y\rangle$ is in $\mathcal A_n'$ if and only if there exist  $C\in \mathcal C_{n}$ and $I\in\mathcal I_n$ such that $I\cap C\not = \emptyset$ and one of the following holds:
\par\smallskip\noindent
\begin{enumerate}
	\item $(x\in L_{CI})\wedge (y\in V_{CI})$,
	\item $( x\in L_{CI})\wedge (y\in R_{CI})$, 
	\item $(x\in V_{CI})\wedge (y\in R_{CI})$.
\end{enumerate}
\par\smallskip\noindent
In words, when restricted to $I\cap C$, the relation  $\mathcal A_n'$ proclaims that $L_{CI}$ is "less than" $V_{CI}$ and that $V_{CI}$ is  "less than" $R_{CI}$.
\par\smallskip\noindent
Finally define $\mathcal A_n$, $\mathcal B_n$, and $<_n$ as follows.
$$
\mathcal A_n=\mathcal A_{n-1}\cup \mathcal A_n'\  and\ 
\mathcal B_n= \mathcal B_{n-1}\setminus \{\langle x,y\rangle:\langle x,y\rangle or \langle y,x\rangle \ is \ in \ \mathcal A_n'\},
$$
$$
<_n=\mathcal A_n \cup \mathcal B_n
$$
\par\bigskip\noindent
\par\bigskip\noindent
Let us check P1-P3. Properties P1 and P2 follow from the definition.
\par\smallskip\noindent
\underline {\it Check of P3.} We need to verify irreflexivity, comparability, and transitivity of $<_n$.
\begin{description}
	\item[\rm Irreflexivity] Fix $\langle x,y\rangle\in <_n$. Assume $\langle x,y\rangle\in \mathcal B_n$. By P2, $\langle x,y\rangle\in \mathcal B_{n-1}$. By P3, $x\not = y$. Assume $\langle x,y\rangle\in \mathcal A_{n-1}$. By P3, $x\not = y$. If $\langle x,y\rangle\not \in \mathcal A_{n-1}\cup \mathcal B_n$, then $\langle x,y\rangle\in \mathcal A_n'$. Then $x$ and $y$ satisfy (1), (2), or (3) in the definition of $\mathcal A_n'$ for some $C\in \mathcal C_n$ and $I\in \mathcal I_n$. Since the sets $L_{CI}, V_{CI}, R_{CI}$ are disjoint, we conclude that $x\not = y$.
	\item[\rm Comparability] Fix distinct $x$ and $y$. Since $<_{n-1}$ is an order, we may assume that $\langle x,y\rangle \in <_{n-1}$.  If $\langle x,y\rangle\in \mathcal A_{n-1}$, then, by P1, $\langle x,y\rangle \in \mathcal A_n\in <_n$. Otherwise, $\langle x,y\rangle\in \mathcal B_{n-1}$. If $\langle x,y\rangle\in \mathcal B_n$, then $\langle x,y\rangle \in <_n$. Otherwise, $\langle x,y\rangle\in \mathcal B_{n-1}\setminus \mathcal B_n$. By the definition of $\mathcal B_n$, we conclude that either $\langle x,y\rangle$ or $\langle y,x\rangle$ is in $\mathcal A_n'\in <_n$.
	\item[\rm Transitivity] Fix $\langle x,y\rangle$ and $\langle y,z\rangle$ in $<_n$. We have four cases to consider.
\begin{description}
	\item[\rm Case ($\langle x,y\rangle, \langle y,z\rangle \not \in \mathcal A_n'$)] The assumption implies that 
$\langle  x,y\rangle, \langle y,z\rangle\in \mathcal A_{n-1}\cup \mathcal B_{n-1}=<_{n-1}$. Since $<_{n-1}$ is an order, $\langle x,z\rangle \in <_{n-1}$.
\par\smallskip
Assume first that $\langle x,z\rangle \not \in \mathcal A_n'$. Then $\langle x,z\rangle \in \mathcal B_n\cup \mathcal A_{n-1}\subset <_n$. 

\par\smallskip
Now assume that that $\langle x,z\rangle \in \mathcal A_n'$.
Then there exist $I\in \mathcal I_n$ and $C\in \mathcal C_n$ such that $x,z$ are in $ C\cap I$ and belong to distinct elements of $\mathcal P=  \{L_{CI}, V_{CI}, R_{CI}\}$. Since $C$ is connected with respect to $\mathcal T_{n-1}$ and P7 holds for $n-1$, we conclude that $\mathcal T_0|_C=\mathcal T_{n-1}|_C$. Since $I$ is convex with respect to $<_0$, the set $C\cap I$ is an open and connected subset of $\langle S, <_{n-1}, \mathcal T_{n-1}\rangle$. Since $x,z\in C\cap I$, $x<_{n-1} y$, and $y<_{n-1} z$, we conclude that $y\in C\cap I$. Then either $x$ and $y$, or $y$ and $z$ are separated by $\mathcal P$. Therefore, either $\langle x,y\rangle\in \mathcal A_n'$ or $\langle y,z\rangle\in \mathcal A_n'$, contradicting the case's assumption. Therefore, the inclusion $\langle x,z\rangle \in \mathcal A_n'$ cannot occur.

	\item[\rm Case ($\langle x,y\rangle \not\in \mathcal A_n', \langle y,z\rangle \in \mathcal A_n'$)] Since $\langle y,z\rangle\in \mathcal A_n'$, there exist $C\in \mathcal C_n$ and $I\in \mathcal I_n$ such that $y$ and $z$ are in $C\cap I$ and belong to distinct elements of $ \{L_{CI}, V_{CI}, R_{CI}\}$. Due to similarity in reasoning, we may assume that $y\in L_{CI}$ and $z\in V_{CI}$.

\par\smallskip
Assume first that $x\in C\cap I$. Since $\langle x,y\rangle \not \in \mathcal A'$, we conclude that $x\in L_{CI}$. Then $\langle x,z\rangle\in \mathcal A_n'\subset <_n$.

\par\smallskip
Assume now that $x\not \in C\cap I$. Since $\langle x,y\rangle\not \in \mathcal A_n'$, we conclude that $x<_{n-1}y$. Since $I\cap C$ is $<_{n-1}$-connected, we conclude that $x<_{n-1}z$. Since 
$\{C\cap I: C\in \mathcal C_n, I\in \mathcal I_n, C\cap I\not = \emptyset\}$ is a disjoint family, neither $\langle x,z\rangle$ nor $\langle z,x\rangle$ is in $\mathcal A_n'$. Therefore, $\langle x,z\rangle\in \mathcal A_n\subset <_n$.

	\item[\rm Case ($\langle x,y\rangle \in \mathcal A_n', \langle y,z\rangle \not \in \mathcal A_n'$)] Similar to the previous case.
	\item[\rm Case ($\langle x,y\rangle, \langle y,z\rangle \in \mathcal A_n'$)] Since 
$\{C\cap I: C\in \mathcal C_n, I\in \mathcal I_n, C\cap I\not = \emptyset\}$ is a disjoint family, we conclude that $x,y$, and $z$ belong to the same element of this family -  $C\cap I$. Since $\langle y,z\rangle \in \mathcal A_n'$, we conclude that $y$ cannot be in $R_{CI}$. Since $\langle x,y\rangle \in \mathcal A_n'$, we conclude that $y$ cannot be in $L_{CI}$. Therefore, $x\in L_{CI}, y\in V_{CI}, z\in R_{CI}$. By the definition of $\mathcal A_n'$, we conclude that $\langle x,z\rangle\in \mathcal A_n'\subset <_n$.
\end{description}

\end{description}

\par\bigskip\noindent
Before verifying the remaining properties P4-P7 of $<_n$, let us make two claims for future reference.
\par\bigskip\noindent
{\it Claim 1.}  $L_{CI}<_nV_{CI}<_n R_{CI}$.
\par\noindent
The statement of this claim follows from our word description of $\mathcal A_n'$.
\par\smallskip\noindent
{\it Claim 2.} {\it  If $<_n$ differs from $<_{n-1}$ on $\{x,y\}$, then  $x$ is  in one of $L_{CI}, V_{CI},R_{CI}$ and $y$ is in one of the other two for some $C$ and $I$.}
\par\noindent
To prove the claim, we may assume that $x<_{n-1} y$. If $\langle x,y\rangle$ were in $\mathcal A_{n-1}$, then
$x<_n y$ would have been true by P1. Therefore, $\langle x,y\rangle \in \mathcal B_{n-1}$. Since $\langle x,y\rangle\not\in <_n$, we conclude that $\langle x,y\rangle\not \in \mathcal B_n$. By the definition of $\mathcal B_n$, we have $\langle y,x\rangle\in\mathcal A'_n$.  The definition of $\mathcal A_n'$ implies that $x$ and $y$ are in the described sets. The claim is proved.

\par\bigskip\noindent
For properties P4-P7,  we define a new topology on $S$ as follows:

\par\medskip\noindent
{\it $\mathcal T_n$ is generated by $<_n$-open sets, $<_{n-1}$-convex sets clopen with respect to $\mathcal T_{n-1}$, and
$\{V_{CI}, S\setminus V_{CI} : C\in \mathcal C_n, I\in \mathcal I_n, C\cap I\not = \emptyset\}$. }

\par\medskip\noindent
Let us verify P4-P7.
\par\smallskip\noindent
\underline {\it Check of P4.} If $I$ is not connected with respect to $\mathcal T_{n-1}$, then $I$ contains an $<_{n-1}$-convex set $U$ which  is clopen with respect to $\mathcal T_{n-1}$. By the definition of $\mathcal T_n$, the set $U$ is in $\mathcal T_n$. Otherwise, $I$ is a non-trivial connected set with respect to $\mathcal T_{n-1}$. Then there exists $C\in \mathcal C_n$ such that $I\cap C$ is not empty. Then $U=V_{CI}$ is as desired.
\par\smallskip\noindent
\underline {\it Check of P5.} Assume that $U$ is $<_m$-convex and clopen in $\langle S, \mathcal T_m\rangle$ for some $m<n$. Then, by induction, $U$ is $<_{n-1}$-convex and clopen in $\langle S,\mathcal T_{n-1}\rangle$. By Claim 2, $U$ is $<_n$-convex. By the definition of $\mathcal T_n$, $U$ is clopen in $\langle S,\mathcal T_n\rangle$.
\par\smallskip\noindent
\underline {\it Check of P6.} Note that  $\{V_{CI}: C\in \mathcal C_n, I\in \mathcal I_n, C\cap I\not = \emptyset\}$ consists of $<_n$-convex sets. By Claim 2, every $<_{n-1}$-convex set clopen with respect to $\mathcal T_{n-1}$ is $<_n$-convex. Therefore, $\mathcal T_n$ is generated by $<_n$-convex sets.
\par\smallskip\noindent
\underline {\it Check of P7.} To show $\mathcal T_n\subset \mathcal T_X$, fix $U\in \mathcal T_n$. Note that all $V_{CI}$'s are open in $X$ by construction. If $U$ is  $<_{n-1}$-convex and  clopen with respect to $\mathcal T_{n-1}$, then $U\in \mathcal T_X$ by P5 for $n-1$. Therefore, we may assume that $U=\{x:x<_n a\}$ for some $a\in S$. There are two cases.
\begin{description}
	\item[\rm Case 1] The assumption is that $a\in I\cap C$ for some $I\in\mathcal I_n$ and $C\in \mathcal C_n$. If $a\in L_{CI}$, then $\{x:x<_na\}=\{x:x<_{n-1} a\}\setminus V_{CI}$. Since $V_{CI}$ is clopen in $X$ by construction and $\{x:x<_{n-1} a\}$ is open in $X$ by assumption for $n-1$, the difference is open in $X$ too. Case $a\in V_{CI}$  and $a\in R_{CI}$ are handled similarly. 
	\item[\rm Case 2] The assumption is that  $a\not \in I\cap C$ for any $I\in\mathcal I_n$ and $C\in \mathcal C_n$. By Claim 2, $\{x:x<_na\}=\{x:x<_{n-1}a\}$. The right side is open in $X$ by assumption.
\end{description}
To prove that $\mathcal T_{n-1}\subset \mathcal T_n$, fix $U\in \mathcal T_{n-1}$. We may assume that $U$ is $<_{n-1}$-convex. If $U$ is clopen with respect to $\mathcal T_{n-1}$, then $U\in \mathcal T_n$ by definition. Therefore, we may assume that $U=\{x:x<_{n-1}a\}$ for some $a$. We have two cases.
\begin{description}
	\item[\rm Case 1] The assumption is that $a\in I\cap C$ for some $I\in\mathcal I_n$ and $C\in \mathcal C_n$. If $a\in L_{CI}$, then $\{x:x<_{n-1}a\}=\{x:x<_{n} a\}\cup \{x\in V_{CI}: x<_{n-1}a\}$. Since $<_n$ and $<_{n-1}$ coincide on $V_{CI}$, we conclude that $\{x\in V_{CI}: x<_{n-1}a\}$ is in $\mathcal T_n$. Hence, the union is in $\mathcal T_n$. Case $a\in V_{CI}$  and $a\in R_{CI}$ are handled similarly. 
	\item[\rm Case 2] The assumption is that  $a\not \in I\cap C$ for any $I\in\mathcal I_n$ and $C\in \mathcal C_n$. By Claim 2, $\{x:x<_{n-1}a\}=\{x:x<_{n}a\}$. The right side is in $\mathcal T_n$ by definition.
\end{description}
\par\smallskip\noindent
\underline {\it Check of P8.} Note that $\langle S,<_n,\mathcal T_n\rangle$ is obtained from $\langle S, <_{n-1}, \mathcal T_{n-1}\rangle$ by lifting $V_{CI}$ and then inserting it between $L_{CI}$ and $R_{CI}$ for each $C\in \mathcal C_n$ and $I\in \mathcal I_n$ with non-empty intersection. Since both $\mathcal C_n$ and $\mathcal I_n$ are disjoint collections of open sets in  $\langle S, <_{n-1}, \mathcal T_{n-1}\rangle$, we conclude that $\langle S,<_n,\mathcal T_n\rangle$ and $\langle S, <_{n-1}, \mathcal T_{n-1}\rangle$ have the same density, and hence weight. The latter has the same weight as $L$ by P8 for $n-1$.

\par\bigskip\noindent
The inductive construction is complete. Put $<_M=[\bigcup_n \mathcal A_n]\cup [\bigcap_n \mathcal B_n]$.
\par\bigskip\noindent
{\it Claim 3.} {\it $<_M$ is an order on $S$.}
\par\smallskip\noindent
By Lemma \ref{lem:limitorder}, to prove the claim it suffices to show that 
$$
<_M=\{\langle x,y\rangle: x<_n y \ for \ all \  large \ enough \ n\}.
$$ 
To prove the inclusion "$\subset $", fix $\langle x,y\rangle\in <_M$. If $\langle x,y\rangle \in \mathcal A_N$ for some $N$, then $x<_ny$ for all $n>N$. If $\langle x,y\rangle\in \bigcap_n \mathcal B_n$, then $x<_n y$ for all $n>1$. Therefore, $<_M\subset \{\langle x,y\rangle: x<_n y \ for \ all \  large \ enough \ n\}$.

To prove the inclusion "$\supset $", fix $\{x,y,K\}$ such that $x<_n y$ for all $n>K$.  Assume that $\langle x,y\rangle \in \mathcal A_N$ for some $N$. Then $x<_My$. Otherwise, $\langle x,y\rangle\in \mathcal B_n$ for all $n>K$. By P2, $\langle x,y\rangle\in \mathcal B_n$ for all $n\leq K$. Therefore, $\langle x,y\rangle \in \bigcap_n\mathcal B_n\subset <_N$. The claim is proved. 
\par\medskip\noindent
Let $\mathcal T_M$ be the topology generated by $\bigcup_n \mathcal T_n$. 

\par\bigskip\noindent
{\it Claim 4. The weight of $\langle S,\mathcal T_M\rangle$ is equal to that of $L$.}
\par\smallskip\noindent
The statement is a direct corollary of P8 and the definition of $\mathcal T_M$.
\par\bigskip\noindent
{\it Claim 5.} $\mathcal T_L\subset \mathcal T_M\subset \mathcal T_X$.
\par\smallskip\noindent
The statement of the claim  follows from  Property P7 and the definition of $\mathcal T_M$.
\par\bigskip\noindent
{\it Claim 6.} $\langle S, <_M, \mathcal T_M\rangle$ is a zero-dimensional GO-space.
\par\smallskip\noindent
To prove the claim, first observe that $M$ is Hausdorff due to inclusion $\mathcal T_L\subset \mathcal T_M$. It is left to show that for any $U\in \mathcal T_M$ and $x\in U$, there exists a $<_M$-convex set $O$ that is clopen in $M$ and $x\in O\subset U$.

Since $\mathcal T_M$ is generated by $\bigcup_n \mathcal T_n$, we may assume that $U\in \mathcal T_n$ for some $n$.
Let $C_x$ be a maximal  connected subset of $U$ with respect to $\mathcal T_n$ that contains $x$. We have the following cases:
\begin{description}
	\item[\rm Case 1] The assumption is  $C_x=\{x\}$.  Since $\langle S,<_n,\mathcal T_n\rangle$ is a GO-space, there exists a $<_n$-convex set $O$ which is a clopen with respect to $\mathcal T_n$ and $x\in O\subset U$. Then, by P5, $O$ is $<_k$-convex and clopen with respect to $\langle S,<_k,\mathcal T_k\rangle$ for any $k\geq n$. Therefore, $O$ is $<_M$-convex and clopen in $M$.
	\item[\rm Case 2] The assumption is that there exist $a$ and $b$ in $C_x$ such that $a<_n x<_n b$. Since $\mathcal T_L\subset \mathcal T_n$ and $[a,b]_{<_n}$ is compact in $\mathcal T_n$, we conclude that $(a,x)_{<_n}$ and $(x,b)_{<_n}$ are non-empty open sets in $L$. Therefore, there exists $I\in \mathcal T_k$ for some $k$ such that $I\subset (a,x)_{<_n}$. By P4 and P6, there exists a non-empty clopen set $V_a$ in $\mathcal T_k$ such that $V_a$ is $<_k$-convex and $V_a\subset I$. Fix $a'\in V_a$. 
Similarly find a $V_b\subset (x,b)_{<_n}$ that is nonempty,  clopen, and convex in some $\langle S,<_m,\mathcal T_m\rangle$.
Fix $b'\in V_b$. Put $n^*=\max\{k,m\}$. The set $(a',b')_{<_n}\setminus [V_a\cup V_b]$ contains $x$, and,  by P5 and P7, is clopen with respect to $\mathcal T_{n^*}$. Since $\langle S, <_{n^*}, \mathcal T_{n^*}\rangle$ is a GO-space, there exists a clopen convex neighborhood $O$ of $x$ in $\langle S, <_{n^*}, \mathcal T_{n^*}\rangle$ that contains $x$ and is contained in $(a',b')_{<_n}\subset U$. Therefore, $O$ is convex and clopen in $M$. Hence, $O$ is as desired.

	\item[\rm Case 3] The assumption is that neither Case 1 nor Case 2 takes place. Then one of end-points of $C_x$ is $x$.
Then we  proceed as in Case 2.
\end{description}
\par\medskip\noindent
By Claims 3-6, the space $M=\langle S, <_M, \mathcal T_M\rangle$ and $h=g=id_S$ are as desired.
\end{proof}

\par\bigskip
We would like to finish with a few questions that naturally arise from our discussion.

\par\bigskip\noindent
\begin{que}
If $L$ is linearly ordered in Theorem \ref{thm:zerodim},  can $M$ be made linearly ordered too?
\end{que}

\par\bigskip\noindent
\begin{que}
Is there a GO-space outside of class $\mathcal L$ for which the conclusion of the theorem fails.
\end{que}

\par\bigskip\noindent
\begin{que}
Let $f:X\to L\in \mathcal L$ be a continuous surjection and $Ind(X)=0$. Does there exists a factorization for $f$ with  a  zero-dimensional GO-space as a middle space?
\end{que}

\par\bigskip\noindent
{\bf Acknowledgment.} The authors would like to thank David Lutzer \cite{Luz} for useful references and the referee for careful reading of the manuscript and many valuable suggestions.

\par

\end{document}